\documentclass[11pt]{article}
\usepackage{amsmath,amssymb,amsthm, latexsym,color,epsfig,enumerate,a4, graphicx}
\usepackage{changepage}

\usepackage{tikz}
\usetikzlibrary{patterns}
\usetikzlibrary{decorations.pathreplacing}

\usepackage{xspace}
\usepackage{bbding}
\usepackage[usenames,dvipsnames]{pstricks}
\usepackage[ruled,vlined, linesnumbered]{algorithm2e}

\theoremstyle{plain}
\newtheorem{theorem}{Theorem}[section]
\newtheorem{lemma}[theorem]{Lemma}
\newtheorem{claim}[theorem]{Claim}

\theoremstyle{definition}

\newcommand{\calA}{\ensuremath{\mathcal A}}
\newcommand{\calB}{\ensuremath{\mathcal B}}

\newcommand{\calE}{\ensuremath{\mathcal E}}

\newcommand{\calH}{\ensuremath{\mathcal H}}

\newcommand{\calS}{\ensuremath{\mathcal S}}

\date{}
\title{\vspace{-0.7cm}Star-factors in graphs of high degree}
\author{
Rajko Nenadov
	\thanks{
		School of Mathematical Sciences, Monash University, Melbourne, Australia. Email: {\tt rajko.nenadov@monash.edu}.
	} 
}

\begin{document}
\maketitle

\begin{abstract}
We prove that every graph with sufficiently large minimum degree $d$ contains a spanning forest in which every component is a star of size at least $\sqrt{d} -  \tilde O(d^{1/4})$. This improves the result of Alon and Wormald and is optimal up to the lower order term.
\end{abstract}

\section{Introduction}

Given graphs $G$ and $H$, a collection of pairwise vertex-disjoint copies of $H$ in $G$ (not necessarily induced) is called an \emph{$H$-packing}. This notion generalises matchings (i.e. $H = K_2$) to arbitrary configurations. A \emph{perfect} $H$-packing (or an \emph{$H$-factor}) is a packing which covers all the vertices of $G$. Starting with the seminal paper of Corr\'adi and Hajnal \cite{hajnal} and its extension by Hajnal and Szemer\'edi \cite{hajnal70sz}, the problem of determining the minimum degree condition which suffices for the existence of an $H$-factor has received considerable attention. After a series of papers (\cite{alon1992almosth,alon1996h,komlos2000tiling,komlos2001proof} to name a few), this line of research has culminated with the work of K\"uhn and Osthus \cite{kuhn2009minimum}, who resolved the problem for every $H$ up to an additive constant term.

Instead of asking for copies of the same graph $H$, one can ask for copies of graphs from some family $\mathcal{H}$. More precisely, a collection of vertex-disjoint subgraphs $H_1, \ldots, H_t \subseteq G$ is an \emph{$\calH$-packing} if $H_i \in \calH$ for every $i$. An \emph{$\calH$-factor} is defined analogously.

In this paper we are interested in the family of large \emph{stars}. A star of size $k$ is a complete bipartite graph with $k$ vertices on one side (called \emph{leaves}) and a single vertex on the other (called the \emph{centre}). Moreover, let $\calS_\ell$ denote the family of all stars of size at least $\ell$. In connection with the analysis of certain exponential time algorithms (see \cite{havet2011exact}), Alon and Wormald \cite{alon2010high} considered the problem of finding the largest $\ell = \ell(G)$ such that a graph $G$ contains an $\calS_\ell$-factor. Their main result states that $\ell = \Omega((d / \log d)^{1/3})$, where $d$ is the minimum degree of $G$. Here we improve upon this bound.

\begin{theorem} \label{thm:main}
	There exist $d_0, C \in \mathbb{N}$ such that every graph $G$ with minimum degree $d \ge d_0$ contains an $S_\ell$-factor for $\ell \ge \sqrt{d} - C d^{1/4} \sqrt{ \log d}$.
\end{theorem}

We make no effort to optimise either $d_0$ or $C$. The next theorem shows that this is almost the best possible without additional assumptions on the graph $G$. 

\begin{theorem} \label{thm:lower}
	For every $d \in \mathbb{N}$ and $n$ sufficiently large there exists a graph $G$ with $n (1 + \lceil \sqrt{d} \rceil) + d$ vertices and minimum degree $d$ which does not contain an $\calS_\ell$-factor  for $\ell > \lceil \sqrt{d} \rceil + 1$.
\end{theorem}

In the case where $G$ is a $d$-regular graph, Alon and Wormald \cite{alon2010high} showed that there exists an $S_\ell$-factor for $\ell = \Omega(d / \log d)$. By considering the minimum size of a dominating set in random $d$-regular graphs, they also showed that this bound is optimal.

%The existence of an $\{S_1, \ldots, S_\ell\}$-factor was considered by Amahashi and Kano \cite{amahashi1982factors}. 
The related problem of \emph{decomposing} graphs into the minimum number of edge-disjoint $\calS_0$-factors (note that $S_0$ is just an isolated vertex), also known as the \emph{star-arboricity}, was first studied by Akiyama and Kano \cite{akiyama85}. The optimal bound on such a number in terms of the maximum degree of a graph was subsequently determined by Algor and Alon \cite{algor1989star} and Alon, McDiarmid, and Reed \cite{alon1992star}. As proposed in \cite{alon2010high}, instead of asking for the minimum number of $\calS_0$-factors it would be interesting to determine if there exists a growing function $\ell$ such that every graph $G$ with minimum degree $d$ can be decomposed into $\calS_{\ell(d)}$-factors.

The next section sets out the tools used in the proof of Theorem \ref{thm:main}, which is presented in Section \ref{sec:main_thm}. Section \ref{sec:lower_bound} gives the lower-bound construction of Theorem \ref{thm:lower}. Whenever the use of floors and ceilings is not important they will be omitted.

\section{Preliminaries}

Similarly to other proofs in this line of research, our main ingredient is the Lov\'asz Local Lemma (see \cite{spencer1977asymptotic}). For our purposes its simplest form suffices.

\begin{lemma}
	Let $\calA_1, \ldots, \calA_n$ be events in an arbitrary probability space. Suppose that each event $\calA_i$ is mutually independent of the set of all the other events but at most $d$, and $\Pr[\calA_i] \le p$ for all $1 \le i \le n$. If
	$$
		ep(d+1) \le 1
	$$ 
	then $\Pr[ \bigwedge_{i = 1}^{n} \overline{\calA_i} ] > 0$.
\end{lemma}

The following claim follows easily from Hall's matching criteria, thus we omit the proof.

\begin{claim} \label{claim:hall}
	Let $G = (V_1 \cup V_2, E)$ be a bipartite graph such that $\deg(v) \ge d_1$ for every $v \in V_1$ and $1 \le \deg(w) \le d_2$ for every $w \in V_2$, for some $d_1 \ge d_2$. Then $G$ contains an $S_\ell$-factor for $\ell = \lfloor d_1 / d_2 \rfloor$, with all centres being in $V_1$.
\end{claim}

\section{Proof of Theorem \ref{thm:main}} \label{sec:main_thm}

In order to point out the difficulties, it is instructive to first look at the case where $G$ is a $d$-regular graph. We follow the proof of Alon and Wormald \cite{alon2010high}. First, we choose each vertex to be in a set $V_1$ with probability $\Theta(\log d / d)$, independently of all other vertices. We want to show that, with a positive probability,  the resulting set $V_1$ has the property that every vertex $v \in V(G)$ has at least one and at most $O(\log d)$ neighbours in $V_1$. In particular, the positive probability implies that such set $V_1$ indeed exists and the desired $\calS_\ell$-factor can be obtained using Claim \ref{claim:hall} with $V_2 = V(G) \setminus V_1$, $d_1 = d - O(\log d)$ and $d_2 = O(\log d)$.

Let us briefly discuss why can we expect that a randomly chosen set $V_1$ has such a property. First, since each vertex has degree $d$, the probability that it does not have a neighbour in $V_1$ is $(1 - \Theta(\log d / d))^{d} < d^{-\Theta(1)}$. On the other hand, Chernoff's inequality tells us the probability that it has significantly more than the expected number of neighbours, which is of order $\log d$, is exponentially small in $\Omega(\log d)$ and thus of order $d^{-\Omega(1)}$. Therefore, a typical vertex fails to have the desired property with probability $d^{-C}$ for some constant $C > 0$ of our choice. Note that if two vertices $v_1, v_2 \in V(G)$ do not have a common neighbour, then the events  `$v_1$ fails' and `$v_2$ fails' are pairwise independent. It is not difficult to see that $v_1$ is actually mutually independent of the set of all such events except those that involve a vertex which has a common neighbour with $v_1$. From the fact that the graph is $d$-regular we deduce that there are at most $d^2$ such events. The conclusion now follows from the Lov\'asz Local Lemma.

However, if the degrees are not equal then this strategy fails to produce a spanning bipartite graph in which the vertices on one side have degree close to $d$ while the vertices on the other have significantly smaller (though still non-zero). Rather than finding such a spanning bipartite graph at once, we proceed in steps and cover $G$ piece by piece. In the first step we find a star-packing $M_1$ which covers all the vertices of very high degree. Crucially, instead of removing the whole such packing from $G$ we first randomly prune its set of leaves to obtain a packing $M_1'$ such that each star is still sufficiently large and each vertex from $V(G) \setminus M_1$ has `good' degree into $G' := G \setminus M_1'$. Note that this is necessary as the graph $G \setminus M_1$ might contain isolated vertices (for example, if there exists a vertex whose neighbourhood is completely contained in the set of leaves from $M_1$).

Note that in the subsequent steps we only have to find a packing of $G'$ which covers $V(G) \setminus M_1$, rather than a factor which covers the whole $G'$. This follows from the simple observation that any uncovered vertex from $M_1 \setminus M_1'$ can be assigned back to the packing $M_1'$. In other words, we have  relaxed the problem from finding a spanning subgraph with certain properties to finding an almost-spanning one which covers certain subset. This is common to many proofs which deal with embeddings of spanning structures. In the next step we take care of the vertices of $V(G) \setminus M_1$ which do not have large degree into $M_1 \setminus M'$ (this is done in the Phase I of Lemma \ref{lemma:aux_cover}). Again, in order to be able to continue we randomly prune the set of leaves of the newly obtained packing. Finally, in the last step we cover the remaining vertices. This is done in the Phase II of Lemma \ref{lemma:aux_cover} and relies on Lemma \ref{lemma:bipartite}. We now make this precise.

\begin{lemma} \label{lemma:bipartite}
	Let $G = (V_1 \cup V_2, E)$ be a bipartite graph with $\deg(v) \ge d$ for every $v \in V_1$. Then $G$ contains an $\calS_{\lfloor \sqrt{d} \rfloor}$-packing which covers $V_1$.
\end{lemma}
\begin{proof}
	Consider a maximal (under vertex-inclusion) $\calS_{\lfloor \sqrt{d} \rfloor}$-packing $M$ with all centres being in $V_2$ (i.e. there is no such packing $M'$ with $M \subset M'$). First, note that there is no edge from $V_1 \setminus V(M)$ to $V_2 \cap V(M)$: otherwise, by appending an endpoint of such an edge from $V_1 \setminus V(M)$ to the corresponding star in $M$ we get a contradiction with the maximality of $M$. Therefore $\deg(v, V_2 \setminus V(M)) = \deg(v, V_2) \ge d$ for every $v \in V_1 \setminus V(M)$. On the other hand, every vertex $u \in V_2 \setminus V(M)$ has degree less than $\lfloor \sqrt{d} \rfloor$ into $V_1 \setminus V(M)$ as otherwise we could extend $M$ by a star of size $\lfloor \sqrt{d} \rfloor$ centred in such $u$. It follows now from Claim \ref{claim:hall} that the bipartite subgraph induced by $V_1 \setminus V(M)$ and $V_2 \setminus V(M)$ contains an $\calS_{\lfloor \sqrt{d} \rfloor}$-packing  which saturates $V_1 \setminus V(M)$. Together with $M$ this gives the desired star matching.
\end{proof}

The following lemma carries out the second and third step outlined in the strategy. Note that the assumption (a) is tailored to the later application of the lemma.

\begin{lemma} \label{lemma:aux_cover}
	Let $G$ be a graph with the maximum degree at most $d^5$ and suppose $S \subseteq V(G)$ is such that 
	\begin{enumerate}[(a)]
		\item $\deg(v) \ge d - 5$ for every $v \in V(G) \setminus S$, and
		\item $\deg(w) \le d$ for every $w \in S$.
	\end{enumerate}
	Then $G$ contains an $\calS_\ell$-packing which covers $V(G) \setminus S$, where
	$$
		\ell = \sqrt{d - 42 d^{3/4} \sqrt{\log d}}.
	$$
\end{lemma}

\begin{proof}
	Without loss of generality we may assume that $G$ is edge-minimal with respect to the property (a), that is, for every edge $e \in G$ there exists $v \in V(G) \setminus S$ such that $\deg_{G \setminus e}(v) < d - 5$. This implies that if two vertices are adjacent in $G$ then at least on of them is of degree exactly $d - 5$.

	The proof of the lemma relies on somewhat delicate partition of $V(G) \setminus S$, which we describe next. First, we partition $V(G) \setminus S$ into vertices of degree exactly $d - 5$ (the set $D$) and vertices of higher degree (the set $H$),
	\begin{align*}
		D &:= \{v \in V(G) \setminus S \colon \deg(v) = d - 5 \}, \\
		H &:= \{v \in V(G) \setminus S \colon \deg(v) > d - 5 \}.
	\end{align*}
	From the edge-minimality of $G$ we have that $H$ is an independent set.	Next, let $D'$ denote the set of vertices from $D$ with significant number of neighbours in $H$,
	$$
		D' := \{v \in D \colon \deg(v, H) \ge d^{3/4} \sqrt{\log d} \}.
	$$
	Finally, we partition $H \cup (D \setminus D')$ depending on the number of neighbours in $D' \cup S$ as follows,
	\begin{align*}
		A &:= \{v \in H \cup (D \setminus D') \colon \deg(v, D' \cup S) \ge d - 3d^{3/4} \sqrt{\log d}\}, \\
		B &:= (H \cup D) \setminus (D' \cup A).
	\end{align*}		
	As each vertex in $H \cup D$ has degree at least $d - 5$, we have
	\begin{alignat}{2}
		\deg(v, D \setminus D') &\ge d^{3/4} \sqrt{\log d}  &&\quad \text{ for every } \quad v \in B, \label{eq:B}\\
		\deg(v, D \cup S) &\ge d - 5 - d^{3/4} \sqrt{\log d} &&\quad \text{ for every } \quad v \in A \cup B. \label{eq:AB}
	\end{alignat}

	\begin{figure}[h!]
		\label{fig:partition}
		\center
		\begin{tikzpicture}[scale = 0.6]

			\draw [decorate,decoration={brace,amplitude=10pt, mirror}] (0,0) -- (16,0);
			\draw [rounded corners] (0,0) rectangle (16,4);		% D
			\node at (8,-1) {\Large $D$};

			\draw [decorate,decoration={brace,amplitude=10pt, mirror}] (16.3,0) -- (22.3,0);
			\draw [rounded corners] (16.3,0) rectangle (22.3,4); 	% S
			\node at (19.3,-1) {\Large $S$};

			\draw [decorate,decoration={brace,amplitude=10pt}] (0.1,7.6) -- (13.9,7.6);
			\draw [rounded corners] (0,4.5) rectangle (14, 7.5);	% H
			\node at (7,8.6) {\Large $H$};

			\node at (11.3,0.5) {$D'$};	% D'
			\draw [thick, rounded corners] (10.7,0) rectangle (16,4);

			\draw [thick, rounded corners] (5.31, 0) -- (10.56, 0) -- (10.58, 4.1) -- (14, 4.5) -- (14, 7.5) -- (7.1, 7.5) -- (7.1, 4.5) -- (5.31, 4) -- cycle; 
			\draw (13.2,7) node[fill=white]{$A$};	% A

			\draw [thick, rounded corners] (0, 0) -- (5.2, 0) -- (5.2, 4.1) -- (6.95, 4.5) -- (6.95, 7.5) -- (0.0, 7.5) -- cycle; 
			\draw (0.6,7) node[fill=white]{$B$};	% B
		
			% vertex in D'
			\draw [thick,fill=white] (7.5,5.5) ellipse (2.2 and 0.9);			
			\draw [thick,fill=white] (13,3) -- (7,5.5) -- (8.5,5.6) -- cycle;			
			\draw (7.5, 5.5) node[fill=white]{$d^{3/4} \sqrt{\log d}$};
			\draw (13,3) node[fill=black,circle,inner sep=2pt,minimum size=1pt] {};

			% vertex in A \cap H
			\draw [thick,fill=white] (16.7,2.9) ellipse (2.2 and 0.9);
			\draw [thick,fill=white] (12.5,5.5) -- (16,3) -- (17,3) -- cycle;
			\draw (16.7, 2.9) node[fill=white]{$(1 - o(1))d$};
			\draw (12.5,5.5) node[fill=black,circle,inner sep=2pt,minimum size=1pt] {};			

			% vertex in A \cap D
			\draw [thick,fill=white] (15.5,1) ellipse (2.2 and 0.9);
			\draw [thick,fill=white] (9,1.5) -- (15,0.5) -- (15.5,1) -- cycle;
			\draw (15.5, 1) node[fill=white]{$(1 - o(1))d$};
			\draw (9,1.5) node[fill=black,circle,inner sep=2pt,minimum size=1pt] {};

			% vertex in B \cap H
			\draw [thick,fill=white] (6.5,2.9) ellipse (2.2 and 0.9);
			\draw [thick,fill=white] (4,5.5) -- (6,2.8) -- (7,2.8) -- cycle;
			\draw (6.5, 2.9) node[fill=white]{$d^{3/4} \sqrt{\log d}$};
			\draw (4,5.5) node[fill=black,circle,inner sep=2pt,minimum size=1pt] {};

			% vertex in B \cap D
			\draw [thick,fill=white] (4.8,1.1) ellipse (2.2 and 0.9);
			\draw [thick,fill=white] (2,3) -- (4,1) -- (4.5,1.5) -- cycle;
			\draw (4.8, 1.1) node[fill=white]{$d^{3/4} \sqrt{\log d}$};	
			\draw (2,3) node[fill=black,circle,inner sep=2pt,minimum size=1pt] {};		
		\end{tikzpicture}
		\caption{Partition of the graph $G$.}
	\end{figure}
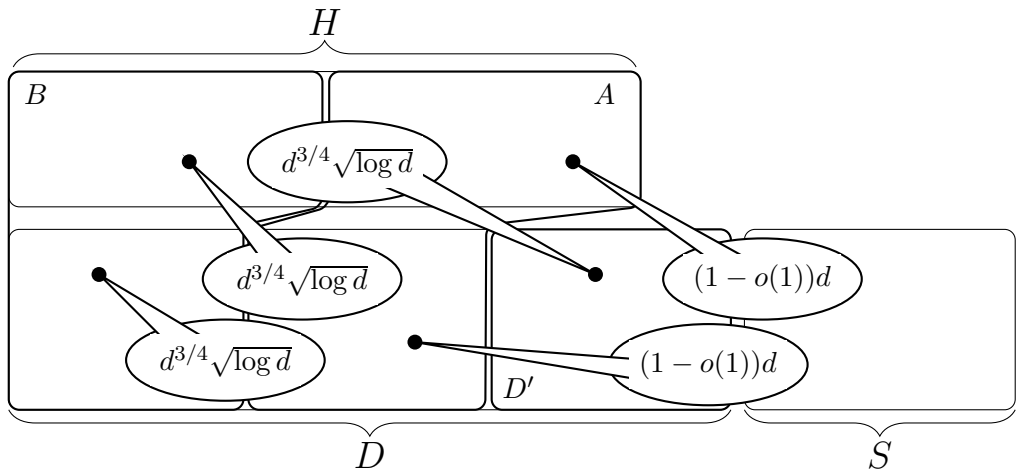

	The proof of the lemma goes in two phases. 

	\medskip
	\noindent
	\textbf{Phase I: covering $B \cup D'$.} In the first phase we find an $\calS_k$-packing which covers $B \cup D'$, where
	$$
		k \ge \frac{d^{3/4}}{10 \sqrt{\log d}},
	$$
	with all centres of stars being in $A \cup B$. We use similar strategy as described in the beginning of this section.

	To this end, let $V_1 \subseteq A \cup B$ be a random subset obtained by choosing each vertex with probability $(8 \sqrt{\log d}) / d^{3/4}$, independently of all other vertices. Using the Lov\'asz Local Lemma, we show that with positive probability $V_1$ does not have any of the following properties:
	\begin{enumerate}[($\calE_1(v)$)]
		\item a vertex $v \in B \cup D'$ has no neighbour in $V_1$, and
		\item a vertex $v \in V(G)$ has at least 
		$$
			\max\{ d, \deg(v) \} \cdot \frac{9 \sqrt{\log d}}{d^{3/4}}
		$$
		neighbours in $V_1$.		
	\end{enumerate}
	Note that an event $\calE_i(v)$ is mutually independent of the set of all events $\calE_j(u)$ except those such that $v$ and $u$ have a common neighbour. Since $H$ is an independent set and all other vertices have degree at most $d$, this implies that $\calE_i(v)$ is mutually independent of the set of all but at most $d^6$ events. Therefore, in order to applying the Local Lemma it suffice to show $\Pr[\calE_i(v)] < 1/d^8$.

	From the definition of $D'$ and \eqref{eq:B} we have that every vertex $v \in B \cup D'$ has at least $d^{3/4} \sqrt{\log d}$ neighbours in $A \cup B$. The probability that it has no neighbours in $V_1$ is at most 
	$$
		\Pr[\calE_1(v)] \le \left(1 - \frac{8 \sqrt{\log d}}{d^{3/4}}\right)^{d^{3/4} \sqrt{\log d}} < e^{- 8 \log d} = 1 / d^8.
	$$
	Next, note that the number of neighbours of a vertex $v \in V(G)$ in $V_1$ is binomially distributed with the expected value of at most
	\begin{equation} \label{eq:expexted}
		\mathbb{E}[\deg(v, V_1)] \le \deg(v) \frac{8\sqrt{\log d}}{d^{3/4}}.
	\end{equation}
	Therefore the desired bound easily follows from Chernoff's inequality,
	$$
		\Pr\left[ \deg(v, V_1) \ge  \max\{d, \deg(v)\} \frac{9 \sqrt{\log d}}{d^{3/4}} \right] \le e^{-\Omega(\max\{d, \deg(v)\} \frac{\sqrt{\log d}}{{d^{3/4}}})} \le e^{- \Omega(d^{1/4} \sqrt{\log d}) }< 1/d^8.
	$$
	This shows that, with positive probability, $V_1$ is such that none of the events $\calE_i(v)$ occur. As a consequence there exists a particular set $V_1$ for which this is the case.

	To summarise, we showed that there exists $V_1 \subseteq A \cup B$ such that
	\begin{itemize}
		\item $\deg(v, V_1) \le 9d^{1/4} \sqrt{\log d}$ for all $v \in D \cup S$ (follows from $\overline{\calE_2(v)}$ and the maximum degree of vertices in $D \cup S$), 
		\item $\deg(v, (D \cup S) \setminus V_1) \ge 0.99d$ for all $v \in A \cup B$ (follows from \eqref{eq:AB} and $\overline{\calE_2(v)}$), and
		\item each vertex from $B \cup D'$ has a neighbour in $V_1$ (follows from $\overline{\calE_1(v)}$).
	\end{itemize}
	Let $V_2 \subseteq (D \cup S) \setminus V_1$ denote the set of vertices with at least one neighbour in $V_1$. From the last property we have
	$$
		((B \cap D) \cup D') \setminus V_1 \subseteq V_2.
	$$
	We can now apply Claim \ref{claim:hall} on the bipartite graphs induced by $V_1$ and $V_2$, with $0.99d$ (as $d_1$) and $9d^{1/4} \sqrt{\log d}$ (as $d_2$). This gives us an $\calS_k$-factor of such bipartite graph with each vertex from $V_1$ being the centre of a star. Finally, as each vertex in $(B \cap H) \setminus V_1$ has a neighbour in $V_1$, we can assign it to the corresponding star. This gives the desired $\calS_k$-packing which covers $B \cup D'$, with all centres being in $V_1 \subseteq A \cup B$.

	\medskip
	\noindent
	\textbf{Phase II: covering $A$.} Let $C \subseteq A \cup B$ denote the set of centres and $L$ the set of leaves in the $\calS_k$-packing obtained in the previous phase. We aim to use Lemma \ref{lemma:bipartite} to cover the vertices in $V_1 := A \setminus (C \cup L)$ (this set should not be confused with the set $V_1$ from the previous phase). In order to avoid any clashes with the previously defined packing we would ideally like to apply Lemma \ref{lemma:bipartite} on the bipartite graph induced by $V_1$ and $(D' \cup S) \setminus L$. However, even though every vertex in $V_1$ has sufficiently large degree into $D' \cup S$ (by the definition of $A$), it might be that for some of them all the neighbours are contained in $L$. In other words, we have no guarantee on the minimum degree from the side of $V_1$. We overcome this by `borrowing' some vertices from $L$ (similar idea is used later in the proof of Theorem \ref{thm:main}).

	For each $u \in C$, let $L(u) \subseteq L \cap (D' \cup S)$ denote the set of leaves from $D' \cup S$ which belong to the star centred in $u$. We show that there exists a subset $L' \subseteq D' \cup S$ such the following holds:
	\begin{enumerate}[(i)]
		\item $|L(u) \cap L'| \ge 2\sqrt{d}$ for all $u \in C$, and
		\item $\deg(v, (D' \cup S) \setminus L') \ge d - 42 d^{3/4} \sqrt{\log d}$ for all $v \in V_1$.
	\end{enumerate}
	The first property ensures that we can freely use the vertices from $(D' \cup S) \setminus L'$ without significantly disturbing the packing obtained in the first phase. The second property allows us to apply Lemma \ref{lemma:bipartite} to cover $V_1$. 

	Having such a set $L'$, we finish the proof as follows. First, from each star obtained in the first phase we remove all the vertices except those that belong to $L'$. By the property (i) this results with the $\calS_{(2 \sqrt{d})}$-packing which does not contain any vertex from $(D' \cup S) \setminus L'$. Now apply Lemma \ref{lemma:bipartite} on the bipartite graph induced by $V_1$ and $(D' \cup S) \setminus L'$  (as $V_2$), with $d - 42 d^{3/4} \sqrt{\log d}$ (as $d$). This gives us an $\calS_\ell$-packing of such bipartite graph which covers $V_1$. Finally, we assign all the vertices from $V_2$ which are not part of this packing back to the stars centred in $C$. All together we obtain an $\calS_\ell$-packing which is guaranteed to cover all the vertices from $V(G) \setminus S$, as required.

	It remains to show that such set $L'$ exists. Let $L'$ be a random subset of $L \cap (D' \cup S)$ obtained by choosing each vertex with probability $30d^{-1/4}{\sqrt{\log d}}$, independently of all other vertices. The expected size of $|L(u) \cap L'|$ is at least $3\sqrt{d}$ and the expected number of neighbours of $v \in V_1$ in $L'$ is at most 
	$$
		\mathbb{E}[\deg(v, L')] \le 30 \max\{d, \deg(v)\} d^{-1/4} \sqrt{\log d} = \Omega(d^{3/4} \sqrt{\log d}).
	$$
	Chernoff's inequality shows that the probability of the actual number being smaller (in the first case) or bigger (in the second) by a factor of $\pm 1/3$ is at most $e^{-\Omega(\sqrt{d})}$, with room to spare. Similarly as before, each such event involving a vertex $v$ is mutually independent of all other events except those involving a vertex $w$ which has a common neighbour with $v$, which counts to at most $d^6$ events (similarly as in the Phase I). Therefore, the existence of $L'$ for which none of these events happen follows from the Lov\'asz Local Lemma and it is easy to see that such $L'$ satisfies properties (i) and (ii).
\end{proof}

With Lemma \ref{lemma:aux_cover} at hand we are ready to prove the main theorem.

\begin{proof}[Proof of Theorem \ref{thm:main}]
	Let $G$ be a graph with minimum degree at least $d$. Without loss of generality we may assume that $G$ is edge-minimal, i.e. removal of any edge results in a graph with minimum degree $d - 1$.

	Let $H \subseteq V(G)$ denote the set of vertices of very high degree,
	$$
		H := \{v \in V(G) \colon \deg(v) > d^5 \},
	$$
	and let $W \subseteq V(G)$ be the set of vertices adjacent to $H$. From the edge-minimality of $G$ we have that $H$ is an independent set and every vertex in $W$ is of degree exactly $d$. Therefore, it follows from Claim \ref{claim:hall} applied on the bipartite graph induced by $H$ (as $V_1$) and $W$ (as $V_2$) that there exists an $\calS_{d^4}$-packing which covers $H$ and $W$, with all centres being in $H$. For each $u \in H$, let $L(u)$ be an arbitrary subset of size $d^4$ of the leaves associated with the star centred in $u$.

	Next we wish to use Lemma \ref{lemma:aux_cover} to obtain a star-packing which covers all the remaining vertices $V' := V(G) \setminus (H \cup W)$. To do that, we first `borrow' a subset $S \subseteq W$ such that the following holds:
	\begin{enumerate}[(i)]
		\item $|L(u) \cap S| \le d^4 - d$ for all $u \in H$ (i.e. we do not significantly disturb the already built packing), and 
		\item $\deg(v, V' \cup S) \ge d - 5$ for all $v \in V'$.
	\end{enumerate}
	Assuming that we have such a set $S$, we finish the proof as follows. From each star centred in $H$ we first remove all leaves which belong to $S$. By property (i), this leaves us with an $\calS_{d}$-packing which covers $H \cup (W \setminus S)$ (in particular, these stars are still significantly larger than required). Next, from Lemma \ref{lemma:aux_cover} we have that the induced subgraph $G[V' \cup S]$ contains an $\calS_\ell$-packing, where
	$$
		\ell \ge \sqrt{d - 42 d^{3/4} \sqrt{\log d}} \ge \sqrt{d} - 42 d^{1/4} \sqrt{\log d}, 
	$$ 
	which covers all vertices in $V' \cup S$ save some subset $S' \subseteq S$. Finally, assigning vertices in $S'$ back to stars centred in $H$ we obtain the desired $\calS_\ell$-factor. It remains to show that such $S$ exists.

	First, for each $v \in V'$ such that $\deg(v, V') < d - 5$ choose an arbitrary subset $\Gamma(v) \subseteq N(v, W)$ of size $d - \deg(v, V') \ge 6$ (here, $N(v, W)$ denotes the set of neighbours of $v$ in $W$). For all other $v \in V'$ put $\Gamma(v) := \emptyset$. Let $S \subseteq W$ be a random subset of vertices obtained by choosing each vertex of $W$ with probability $1 - 1/d^2$, independently of all other vertices. For each $u \in H$ let $\calA(u)$ denote the event that $|L(u) \cap S| > d^4 - d$, and for each $v \in V'$ let $\calB(v)$ denote the event that $|\Gamma(v) \setminus S| \ge 6$. We claim that with positive probability none of the events $\calA(u)$ and $\calB(v)$ happen, which implies that there exists $S$ which satisfies (i) and (ii).

	Recall that $|L(u)| = d^4$. We bound the probability $\Pr[\calA(u)]$ as follows,
	$$
		\Pr[\calA(u)] = \Pr[|L(u) \cap S| > d^4 - d] \le \binom{d^4}{d^4 - d} (1 - 1/d^2)^{d^4 - d} \le \binom{d^4}{d} e^{-d^2 + 1/d} < 1/ed^5,
	$$
	where in the last inequality we assumed $d$ is sufficiently large. The bound on $\Pr[\calB(v)]$ is obtained in a similar way,
	$$
		\Pr[\calB(v)] = \Pr[|\Gamma(v) \setminus S| \ge 6] \le \binom{|\Gamma(v)|}{6} (1 / d^2)^6 \le \binom{d}{6} d^{-12} < 1/ed^5.
	$$
	Next, note that $\calA(u)$ is mutually independent of the set of all events $\calA(u')$ and $\calB(v)$ except those $\calB(v)$ for which $L(u) \cap \Gamma(v) \neq \emptyset$. Similarly, $\calB(v)$ is mutually independent of the set of all events $\calA(u)$ and $\calB(v')$ except those that satisfy $\Gamma(v) \cap L(u) \neq \emptyset$ and $\Gamma(v) \cap \Gamma(v') \neq \emptyset$, respectively. Since each vertex in $W$ has degree $d$, from $|\Gamma(v)| < |L(u)| = d^4$ we conclude that every event is mutually independent of the set of all other events but at most $d^5$. Therefore, by the Lov\'asz Local Lemma we have that with positive probability none of the events happen. This concludes the proof.
\end{proof}

\section{The lower-bound construction} \label{sec:lower_bound}

In this section we construct graphs with minimum degree $d$ which do not contain an $\calS_\ell$-factor for $\ell > \lceil \sqrt{d} \rceil + 1$.  As already mentioned, this shows that the bound given by Theorem \ref{thm:main} is optimal up to the lower order term.

\begin{proof}[Proof of Theorem \ref{thm:lower}]
	Let $G$ be any graph with the vertex set $V(G) = A \cup B \cup C$, where 
	$$
		|A| = n, \; |B| = \lceil \sqrt{d} \rceil n \; \text{ and } \; |C| = d,
	$$
	such that the following holds:
	\begin{enumerate}[(i)]
		\item $A$, $B$ and $C$ are independent sets,
		\item every vertex $v \in A$ has $d$ neighbours in $B$ and no neighbours in $C$,
		\item every vertex $w \in B$ has at most $\lceil \sqrt{d} \rceil$ neighbours in $A$,
		\item every two vertices from $B$ and $C$ are adjacent (i.e. $G[B, C]$ is a complete bipartite graph).
	\end{enumerate}
	Clearly, $G$ has minimum degree $d$. Note that such a graph can be constructed sequentially by making a vertex $v \in A$ adjacent to an arbitrary subset of $d$ vertices in $B$ with the smallest degree. 

	We show that $G$ does not contain an $\calS_\ell$ factor for $\ell \ge \lceil \sqrt{d} \rceil + 2$. Note that a star centred in $A$ has all leaves in $B$, thus the maximum number of vertex disjoint stars of size $\ell$ with the centre in $A$ is at most 
	$$
		|B| / (\sqrt{d} + 2) \le \frac{(\sqrt{d} + 1)n}{\sqrt{d} + 2}.
	$$
	On the other hand, as each vertex in $B$ has at most $\lceil \sqrt{d} \rceil$ neighbours in $A$ we conclude that an $S_\ell$ centred in $B$ has at least two leaves in $C$. This implies that there can be at most $d$ such vertex-disjoint stars and their neighbourhood covers at most $d \lceil \sqrt{d} \rceil$ vertices in $A$. However, for $n$ sufficiently large we have 
	$$
		d \lceil \sqrt{d} \rceil + \frac{(\sqrt{d} + 1)n}{\sqrt{d} + 2} < n,
	$$
	and therefore there always exists a vertex in $A$ which is neither a centre nor a leaf. This concludes the proof.
\end{proof}

\bibliographystyle{abbrv}
\bibliography{refs}

\end{document}